\documentclass[11pt,a4paper,fleqn]{article}                               
\usepackage[T1]{fontenc}                   
\usepackage[utf8]{inputenc}                 
\usepackage[english]{babel}       
\usepackage{graphicx}                      % 
\usepackage{caption}        
\usepackage{mathrsfs}          %                  
\usepackage[top=.8in,bottom=1in,left=1in,right=1in]{geometry}
\usepackage[english]{babel}
\usepackage{verbatim}
\usepackage{color}
\usepackage{picture}
\usepackage{amsmath,amsthm,amsfonts,amssymb}
\usepackage{eufrak}
\usepackage{pgf,tikz}
\usetikzlibrary{arrows}
\newtheorem{thm}{Theorem}[section]
\newtheorem{lem}{Lemma}[section]
\newtheorem{cor}{Corollary}[section]
\newtheorem{prop}{Proposition}[section]
\newtheorem{defn}{Definition}[section]
\theoremstyle{definition}
\newtheorem{rem}{Remark}[section]

\theoremstyle{remark}

\newcommand{\D}{\nabla}
\newcommand{\haus}{\mathcal{H}^{n-1}}
\newcommand{\ds}{\displaystyle}
\newcommand{\norm}[1]{\left\Vert#1\right\Vert}

\newcommand{\abs}[1]{\left\vert#1\right\vert}

\newcommand{\R}{\mathbb{R}}

\newcommand{\Rn}{\mathbb R^n}
\newcommand{\W}{\mathcal W}
\newcommand{\de}{\partial}
\newcommand{\eps}{\varepsilon}

\newcommand{\utilp}{\tilde{u}_p}

\newcommand{\fhi}{\varphi}
\DeclareMathOperator{\dive}{div}
\DeclareMathOperator{\Qp}{\mathcal Q_{p}}

\makeatletter %note a di pagina senza numero 1
\def\@makefnmark{} %note a di pagina senza numero 2
\makeatother %note a di pagina senza numero 3
\numberwithin{equation}{section}
\makeatother

\title{On the first Robin eigenvalue of a class of anisotropic operators} 
\author{
   Nunzia Gavitone%
\thanks{
Universit\`a degli studi di Napoli Federico II, Dipartimento di Matematica e Applicazioni ``R. Caccioppoli'', Via Cintia, Monte S. Angelo - 80126 Napoli, Italia. Email: nunzia.gavitone@unina.it}
{, }
Leonardo Trani%
\thanks{Universit\`a degli Studi di Napoli Federico II, Dipartimento di Matematica e Applicazioni ``R. Caccioppoli'', Via Cintia, Monte S. Angelo - 80126 Napoli, Italia. Email: leonardo.trani@unina.it}
}

\begin{document}
\maketitle

\begin{abstract}
The paper is devoted to the study of some properties of the first  eigenvalue of the anisotropic  $p$-Laplace operator with  Robin boundary condition involving a function $\beta$ which in general is not constant. In particular we obtain sharp lower bounds in terms of the measure of the domain and we prove a monotonicity property of the eigenvalue with respect the set inclusion. 
\vspace{.2cm}
\end{abstract}



%%%%%%%%%%%%%%%%%%%%%%%%%%%%%%%%%%%%%%%%%%%%%%%%%%%%%%%%%%%%%%%%%%%
\section{Introduction}

Let $F$ be a  norm in $\R^{n}$, that is  a convex, even, 1-homogeneous and non negative  function defined in $\R^{n}$. Moreover we will assume that  $F\in C^{2}(\mathbb R^{n}\setminus \{0\})$, and strongly convex that is for $1<p<+\infty$, it holds 
\begin{equation*}
\label{strong}
[F^{p}]_{\xi\xi}(\xi)\text{ is positive definite in } \R^{n}\setminus\{0\}.
\end{equation*}
 For $1<p<+\infty$ the so-called anisotropic $p$-Laplacian is defined as follows
\begin{equation*}
\label{qp}
\Qp u:= \dive \left(\frac{1}{p}\nabla_{\xi}[F^{p}](\nabla u)\right).
\end{equation*}
The assumptions on $F$ ensure that  the operator $\Qp$ is elliptic. The paper concerns the study of the following Robin eigenvalue problem for $\Qp$
\begin{equation}\label{intro:bvp.var.beta}
\begin{cases}
-\Qp v=\ell_{1}(\beta,\Omega) \vert v\vert^{p-2}v &\mbox{in}\ \Omega\\[.2cm]
F^{p-1}(\D v)F_\xi(\D v)\cdot\nu + \beta(x) F(\nu)\vert v\vert^{p-2}v=0 &\mbox{on}\ \de\Omega,
\end{cases}
\end{equation}
where $\Omega \subset \R^{n}\ \mbox{is a bounded open set with}\ C^{1,\alpha} \ \mbox{boundary, }\alpha \in ]0,1[,
$ $\nu$ is the Euclidean unit outer normal to $\de\Omega$ and  the function $\beta \colon \de \Omega \to [0,+\infty[$ belongs to $L^{1}(\de \Omega)$ and verifies
\begin{equation*}
\int_{\de \Omega}\beta(x) F(\nu)\,d\mathcal H^{n-1}=m >0.
\end{equation*}
Here $\ell_{1}(\beta,\Omega)$ is the first Robin eigenvalue of $\Qp$ and it has the following variational characterization  
\begin{equation}\label{intro:var.beta}
\ell_1(\beta, \Omega)=\inf_{\substack{v\in W^{1,p}(\Omega)\\ v\not\equiv 0}}\frac{\ds\int_{\Omega} F^p(\D v) dx+\ds\int_{\de\Omega}\beta(x)\vert v\vert^p F(\nu)d\haus}{\ds\int_{\Omega}\vert v\vert^p dx},
\end{equation}
and the minimizers of \eqref{intro:var.beta} are weak solutions to the  problem \eqref{intro:bvp.var.beta} (see section 3 for the precise definition).  
When $F(\xi)=|\xi|$ is the Euclidean norm, this  problem has been studied for instance in \cite{dan,bd10,df,jk,DpGK,gs}. In particular in \cite{bd10} and \cite{df} when $\beta(x)=\bar \beta$ is a nonnegative constant and for any $p$, $1<p<\infty$,  the authors prove a sharp lower bound for $\ell_{1}(\bar\beta, \Omega)$, keeping fixed the measure of the domain $\Omega$.
More precisely, they prove the following Faber-Krahn type inequality
\begin{equation}
\label{intro:fktesi}
 \ell_{1}(\bar\beta,\Omega)\ge \ell_{1}(\bar\beta,B_{R}),
\end{equation}
where $B_{R}$ is a ball having the same measure than $\Omega$. To prove this result they need mainly two key properties of $\ell_{1}(\bar \beta, \Omega),$ that is  a level set representation formula  and the decreasing  monotonicity of  $\ell_{1}(\bar \beta, \Omega)$  with respect to the radius when $\Omega$ is a ball that is
\begin{equation}
\label{gsb}
 \ell_{1}(\bar\beta, B_{r}) \le  \ell_{1}(\bar\beta, B_{s}), \quad r>s>0.
\end{equation}    
Despite to the Dirichlet eigenvalue, in general $\ell_{1}(\bar \beta, \Omega)$ is not monotone decreasing  with respect the set inclusion. For instance, in \cite{gs} when $p=2$ and $\beta=\bar \beta$, the authors prove a sort of  monotonicity property \eqref{gsb} for suitable convex domains which are not necessary balls   and they prove that   
\begin{equation}
\label{gs}
 \ell_{1}(\bar\beta, \Omega_{2}) \le  \ell_{1}(\bar\beta, \Omega_{1}),
\end{equation}
where $\Omega_{1}, \Omega_{2} \subset \R^{n}$ are bounded, Lipschitz and  convex domains such that $\Omega_{1} \subset B_r\subset\Omega_{2}$. 
Our aim is to prove \eqref{intro:fktesi} and \eqref{gs} in the anisotropic case  for any $p$, $1<p<\infty$ and for a suitable function $\beta$ which is  in general, not necessary constant. In particular, regarding  \eqref{intro:fktesi}, we will prove the following anisotropic  Faber-Krahn inequality
\begin{equation}
\label{intro:fktesia}
 \ell_{1}(\beta,\Omega)\ge \ell_{1}(\beta,\mathcal W_{R}),
\end{equation}
where $\mathcal W_{R}=\{F^{o}(\xi)<R\}$, with $F^{o}$ polar norm of $F$, such that $|\mathcal W_{R}|=|\Omega|$ and the function $\beta(x)=w(F^{o}(x))$ with $w$  non negative continuous function in $\R$ such that 
\begin{equation}
\label{ipbeta}
w(t) \ge C(R)\,t,
\end{equation} 
where $C$ is a suitable constant.
To do this we need to establish a representation formula for $\ell_{1}(\beta,\Omega)$, for not constant $\beta$.
As a consequence of this formula, we also obtain the following anisotropic weighted Cheeger inequality for $\ell_{1}(\beta,\Omega)$   
\begin{equation}
\label{intro:caw}
\ell_1(\beta, \Omega)\geq h_{\beta}(\Omega) - (p-1)\|\beta_{\Omega}^{p^\prime}\|_{L^{\infty}(\overline\Omega)}, 
\end{equation}
where $p'=\frac{p}{p-1}$, $\beta_{\Omega}$ is a function defined in the whole $\Omega$ having trace on $\de \Omega$ equals to $\beta$ and $ h_{\beta}(\Omega)$ is the anisotropic weighted Cheeger constant defined in section 6. This result was proved in the Euclidean case in \cite{jk} for $p=2$ and $\beta=\bar \beta$ constant.

Our paper has the following structure.

In section 2 we recall notation and preliminary results. In section 3 we prove some basic properties of $\ell_1(\beta,\Omega)$. In section 4 we prove some useful properties of the anisotropic radial problem. In section 5 we state and show the quoted monotonicity result  for $\ell_1(\beta,\Omega)$ and finally in section 6 we prove the representation formula for level set in the general case of variable coefficients $\beta$ proving as applications the quoted Faber-Krahn inequality \eqref{intro:fktesia} and the anisotropic weighted Cheeger inequality \eqref{intro:caw}.

%%%%%%%%%%%%%%%%%%%%%%%%%%%%%%%%%%%%%%%%%%%%%%%%%%%%%%%%%%%%%%%%%%
\section{Notation and preliminaries}
\label{notation}
\subsection{Finsler norm}
Let $F$ be  a convex, even, 1-homogeneous and non negative  function defined in $\R^{n}$. 
Then $F$ is a convex function such that
\begin{equation}
\label{eq:omo}
F(t\xi)=|t|F(\xi), \quad t\in \R,\,\xi \in  \R^{n}, 
\end{equation}
 and such that
\begin{equation}
\label{eq:lin}
a|\xi| \le F(\xi),\quad \xi \in  \R^{n},
\end{equation}
for some constant $a>0$. The hypotheses on $F$ imply there exists $b\ge a$ such that
\begin{equation}
\label{upb}
F(\xi)\le b |\xi|,\quad \xi \in  \R^{n}.
\end{equation}
Moreover, throughout the paper we will assume that $F\in C^{2}(\mathbb R^{n}\setminus \{0\})$, and
\begin{equation}
\label{strong}
[F^{p}]_{\xi\xi}(\xi)\text{ is positive definite in } \R^{n}\setminus\{0\},
\end{equation}
with $1<p<+\infty$. 

%\begin{rem}
%We stress that for $p\ge 2$ the condition  
%\begin{equation*}
%\nabla^{2}_{\xi}[F^{2}](\xi)\text{ is positive definite in } \R^{N}\setminus\{0\},
%\end{equation*}
%implies \eqref{strong}.
%\end{rem}
The polar function $F^o\colon \R^n \rightarrow [0,+\infty[$ 
of $F$ is defined as
\begin{equation*}
F^o(v)=\sup_{\xi \ne 0} \frac{\langle \xi, v\rangle}{F(\xi)}. 
\end{equation*}
 It is easy to verify that also $F^o$ is a convex function
which satisfies properties \eqref{eq:omo} and
\eqref{eq:lin}. Furthermore, 
\begin{equation*}
F(v)=\sup_{\xi \ne 0} \frac{\langle \xi, v\rangle}{F^o(\xi)}.
\end{equation*}
The above property implies the following anisotropic version of the Cauchy Schwartz inequality
\begin{equation*}
\label{imp}
|\langle \xi, \eta\rangle| \le F(\xi) F^{o}(\eta), \qquad \forall \xi, \eta \in  \R^{n}.
\end{equation*}
The set
\[
\mathcal W = \{  \xi \in  \R^n \colon F^o(\xi)< 1 \}
\]
is the so-called Wulff shape centered at the origin. We put
$\kappa_n=|\mathcal W|$, where $|\mathcal W|$ denotes the Lebesgue measure
of $\mathcal W$. More generally, we denote by $\mathcal W_r(x_0)$
the set $r\mathcal W+x_0$, that is the Wulff shape centered at $x_0$
with measure $\kappa_nr^n$, and $\mathcal W_r(0)=\mathcal W_r$.

%We observe that $F$ is the support function  of $\overline{\mathcal W}$. 

The following properties of $F$ and $F^o$ hold true:
\begin{gather*}
\label{prima}
 \langle F_\xi(\xi) , \xi \rangle= F(\xi), \quad  \langle F_\xi^{o} (\xi), \xi \rangle
= F^{o}(\xi),\qquad \forall \xi \in
 \R^n\setminus \{0\},
 \\
 \label{seconda} F(  F_\xi^o(\xi))=F^o(  F_\xi(\xi))=1,\quad \forall \xi \in
 \R^n\setminus \{0\}, 
\\
\label{terza} 
F^o(\xi)   F_\xi( F_\xi^o(\xi) ) = F(\xi) 
 F_\xi^o\left(  F_\xi(\xi) \right) = \xi\qquad \forall \xi \in
 \R^n\setminus \{0\}. 
\end{gather*}
\subsection{Anisotropic perimeter}
We recall the definition of anisotropic  perimeter for  a bounded, Lipschitz open set:
\begin{defn}\label{aniper}
Let $K$ be a bounded open subset of $\R^n$ with Lipschitz boundary. The anisotropic perimeter of $K$ is:
\[
P_F(K)=\displaystyle \int_{\de K}F(\nu) \, d \mathcal H^{n-1}
\]
where $\nu$ denotes the unit outer normal to $\de K$ and $\haus$ is the $(n-1)$-dimensional Hausforff measure. \end{defn}
Clearly, the perimeter of  $K$ is finite if and only if the usual Euclidean perimeter of $K$, $P_{\mathcal E}(K)$ is finite. Indeed, by the quoted properties of $F$ we obtain that
\[
aP_{\mathcal E}(K) \le P_F(K) \le bP_{\mathcal E}(K).
\] 
Furthermore, an isoperimetric inequality for the anisotropic perimeter holds (see for instance \cite{bu,aflt,fomu}). Namely let K be a bounded open subset of $\R^n$ with Lipschitz boundary, then
\begin{equation}\label{isop_ani}
P_F(K) \ge n \kappa_n^\frac{1}{n} \abs{K}^{1-\frac{1}{n}},
\end{equation}
where $\kappa_n$ is the Lebesgue measure of the unit Wulff shape. In particular, the equality  in \eqref{isop_ani} holds if and only if the set $K$ is homothetic to a Wulff shape.
We recall the following so-called weighted anisotropic isoperimetric inequality (see for instance  \cite{bro} and \cite{bf})
\begin{equation}
\label{wii}
\int_{\de \Omega} f(F^o(x))F(\nu)\, d \mathcal H^{n-1} \ge \int_{\de \mathcal W_R} f(F^o(x))F(\nu) \, d \mathcal H^{n-1}= f(R) P_F(\mathcal W_{R}),
\end{equation}
where $\mathcal W_R$ is a Wulff shape such that $|\Omega|=|\mathcal W_R|$ and $ f\colon [0,R] \to [0,+\infty[$ is a nondecreasing function such that 
\begin{equation*}
g(z)=f(z^{\frac 1 n}) z^{1-\frac{1}{n}} , \qquad  0 \le z \le R^{n}, 
\end{equation*}
is convex with respect to $z$.

If $\Omega \subset \R^{n}$ is a bounded open set, the anisotropic Cheeger constant of $\Omega$ is defined as follows
\begin{equation}
\label{cc}
h_{F}(\Omega)=\inf_{U\subset\Omega}\frac{P_{F}(U)}{\abs{U}}.
\end{equation}
In \cite{DpDbG} the authors prove that
\begin{equation}
\label{hr}
\frac{1}{R_{F}} \le h_{F}(\Omega)\le \frac{n}{R_{F}},
\end{equation}
where  $R_{F}$ is the anisotropic inradius that is the radius of the biggest Wulff shape contained in  $\Omega$.  

%\begin{equation}
%\label{inrad}
%R_{F}=
%\end{equation}
%%%%%%%%%%%%%%%%%%%%%%%%%%%%%%%%%%%%%%%%%%%%%%%%%%%%%%%%%%%%%%%%%
\subsection{Anisotropic $p$-Laplacian}
Let $\Omega\subset \R^{n}$ be a bounded open set and  $u \in W^{1,p}(\Omega)$. For $1<p<+\infty$ the anisotropic $p$ Laplacian is defined as follows
\begin{equation*}
\label{qp}
\Qp u:= \dive \left(\frac{1}{p}\nabla_{\xi}[F^{p}](\nabla u)\right).
\end{equation*}
The hypothesis \eqref{strong} on $F$ ensures that the operator 
is elliptic, hence there exists a positive constant $\gamma$ such that
\begin{equation*}
\frac1p\sum_{i,j=1}^{n}{\nabla^{2}_{\xi_{i}\xi_{j}}[F^{p}](\eta)
  \xi_i\xi_j}\ge
\gamma |\eta|^{p-2} |\xi|^2, 
\end{equation*}
for any $\eta \in \Rn\setminus\{0\}$ and for any $\xi\in \Rn$. 

For $p=2$, $\mathcal Q_{2}$ is the so-called Finsler Laplacian, and when $F(\xi)=|\xi|=\sqrt{\sum_{i=1}^{n}x_{i}^{2}}$ is the Euclidean norm, $\Qp$ reduces to the well known $p$-Laplace operator.  

Let $\Omega$ be a bounded open set in $\R^{n}$, $n\ge 2$, $1<p<+\infty$, and consider the following eigenvalue problem with Dirichlet boundary conditions related to $\Qp$
\begin{equation*}
\label{eigpb}
\left\{
\begin{array}{ll}
-\Qp u=\lambda |u|^{p-2}u & \text{in } \Omega \\
u=0 &\text{on } \partial\Omega.
\end{array}
\right.
\end{equation*}

The smallest eigenvalue, denoted  by $\lambda_{D}(\Omega)$, has the following well-known variational characterization:
\begin{equation*}
\label{rayleigh}
\lambda_{D}(\Omega)=\min_{\varphi\in W^{1,p}_{0}(\Omega)\setminus\{0\}} \frac{\ds\int_\Omega F^p(\nabla \varphi)\ dx}{\ds\int_\Omega |\varphi|^p\ dx}.
\end{equation*}
For the first eigenvalue of the anisotropic $p$-Laplacian with Dirichlet boundary conditions, the following isoperimetric inequality holds (see \cite{bfk}).
\begin{thm}\label{FKdir}
Let $\Omega\subset \R^{n}$, be a bonded domain with $n\ge 2$ then
\begin{equation*}
\abs{\Omega}^{\frac{p}{n}}\lambda_D(\Omega) \ge \kappa_n^{\frac{p}{n}}\lambda_D(\mathcal{W}).
\end{equation*}
Moreover, the equality holds if and only if $\Omega$ is homothetic to a Wulff shape.
\end{thm}
Finally we recall that for a given bounded open set in $\R^{n}$, the anisotropic Cheeger inequality states that (see for instance \cite{cas,DpDbG,kn})
\begin{equation}
\label{cheeger}
\lambda_{D}(\Omega)\ge \left(\frac{h_{F}(\Omega)}{p}\right)^{p}, \quad 1 <p<\infty.
\end{equation}
%
%%%%%%%%%%%%%%%%%%%%%%%%%%%%%%%%%%%%%%%%%%%%%%%%%%%%%%%%%%%%%%%%%%%%%
%
%
%\subsubsection{Some classical embedding theorems}
%Here we recall two well known  embedding  results for Sobolev space which we need in order to prove existence results for the minimum of suitable functionals related to $\Qp$. 
%\begin{thm}[Rellich-Kondrachov]\label{RK}
%   Let $p >1$ and let  $\Omega$ be a bounded open set with Lipschitz boundary Suppose that $u_{k} \in W_{0}^{1,p}(\Omega)$ and that $\|\nabla u_{k}\|_{W_{0}^{1,p}} \le C <+\infty$. Then there exist a subsequence still denoted by $u_{k}$ and a  function $u \in W_{0}^{1,p}(\Omega)$ such that $u_{k} \to u$ strongly in $L^{p}(\Omega)$ and $\nabla u_{k} $ converges weakly in $L^{p}(\Omega)$ to $\nabla u$.
% \end{thm}
%The next result is a compactness embedding theorem for the trace of a Sobolev function.
% \begin{thm}[Compactness of Trace Embeddings]\label{TE}
%   Let $\Omega$ be a bounded open set with Lipschitz boundary, $p>1$.  The Trace mapping $T: W^{1,p}(\Omega) \rightarrow L^q(\de\Omega)$ is compact. In particular the following inequalities hold
%\begin{equation*}\label{tra.ine}
%\Vert v\Vert_{L^q(\de\Omega)}\leq C\Vert v\Vert_{W^{1,p}(\Omega)},\ \mbox{for}\
%\begin{cases}
%q=\frac{p(n-1)}{n-p} &\mbox{if}\ p<n\\[.1cm]
%q<+\infty &\mbox{if}\ p=n\\[.1cm]
%q=+\infty &\mbox{if}\ p>n
%\end{cases}
%\end{equation*}
%\end{thm} 
%%%%%%%%%%%%%%%%%%%%%%%%%%%%%%%%%%%%%%%%%%%%%%%%%%%%%%%%%%%%%%%%%%%%
\section{The first Robin eigenvalue of $\Qp$}
In this section we will investigate some properties of the first Robin eigenvalue related to $\Qp$, $1<p <\infty$.
From now on we assume that  
\begin{equation}\label{dom_cond}
\Omega \subset \R^{n}\ \mbox{is a bounded open set with}\ C^{1,\alpha} \ \mbox{boundary and } \alpha \in ]0,1[.
\end{equation}
Let us consider the following Robin eigenvalue problem for $\Qp$
\begin{equation}\label{bvp.var.beta1}
\begin{cases}
-\Qp u=\ell \vert u\vert^{p-2}u &\mbox{in}\ \Omega\\[.2cm]
F^{p-1}(\D u)F_\xi(\D u)\cdot\nu + \beta(x) F(\nu)\vert u\vert^{p-2}u=0 &\mbox{on}\ \de\Omega,
\end{cases}
\end{equation}
where $u\in W^{1,p}(\Omega)$, $\nu$ is the Euclidean unit outer normal to $\de\Omega$ and  the function $\beta \colon \de \Omega \to [0,+\infty[$ belongs to $L^{1}(\de \Omega)$ and verifies
\begin{equation}
\label{m}
\int_{\de \Omega}\beta(x) F(\nu)\,d\mathcal H^{n-1}=m >0.
\end{equation}
From now on we will write $\bar \beta$ instead of $\beta$ when $\beta$ is a positive constant.
\begin{defn}
\label{aut}
A function $u\in W^{1,p}(\Omega)$, $u \not\equiv 0$ is an eigenfunction  to \eqref{bvp.var.beta1} if  $\beta(\cdot)\vert u\vert^p\in L^1(\de\Omega)$ and
\begin{equation}\label{weak.problem}
\ds\int_{\Omega}F^{p-1}(\D v)F_{\xi}(\D v)\cdot \D\varphi dx + \ds\int_{\de\Omega}\beta(x)\vert u\vert^{p-2} u\varphi F(\nu) d\haus = \ell\ds\int_{\Omega}\vert u\vert^{p-2}u\varphi dx
\end{equation}
for any test function $\varphi\in W^{1,p}(\Omega)\cap L^{\infty}(\de\Omega)$. The corresponding number $\ell$, is called Robin eigenvalue.
\end{defn}
The smallest eigenvalue of \eqref{bvp.var.beta1}, $\ell_{1}(\beta,\Omega)$ has the following variational characterization
\begin{equation}\label{var.beta}
\ell_1(\beta, \Omega)=\inf_{\substack{v\in W^{1,p}(\Omega)\\ v\not\equiv 0}}J[\beta,v]:=\inf_{\substack{v\in W^{1,p}(\Omega)\\ v\not\equiv 0}}\frac{\ds\int_{\Omega} F^p(\D v) dx+\ds\int_{\de\Omega}\beta(x)\vert v\vert^p F(\nu)d\haus}{\ds\int_{\Omega}\vert v\vert^p dx}. 
\end{equation}
By definition we have 
\[
\ell_1(\beta, \Omega) \le \lambda_{D}(\Omega),
\]
where $\lambda_{D}(\Omega)$ is the first Dirichlet eigenvalue of $\Qp$.
Indeed choosing as test function in \eqref{var.beta}, 
 the first Dirichlet eigenfunction $u_D $ of $\lambda_{D}(\Omega)$ in the Reileigh quotient, we get \begin{multline*}
\ell_1( \beta, \Omega) = \min_{\substack{v\in W^{1,p}(\Omega)\\ v\neq 0}}
\frac{\ds\int_{\Omega} [F(\D v)]^{p}dx +\int_{\de\Omega}\beta|v|^{p}F(\nu)d\haus}{\ds\int_{\Omega}|v|^{p}dx}\\[.15cm]
\leq \frac{\ds\int_{\Omega} [F(\D u_D)]^{p}dx +\int_{\de\Omega}\beta|u_D|^{p}F(\nu)d\haus}{\ds\int_{\Omega}|u_D|^{p}dx}= \frac{\ds\int_{\Omega} [F(\D u_D)]^{p}dx}{\ds\int_{\Omega} \abs{u_D}^{p}dx}
=\lambda_D(\Omega).
\end{multline*}

The following  existence result  holds.

\begin{prop}\label{eig_betavar}
Let $\beta\in L^1(\de\Omega)$, $\beta\geq 0$ be such that \eqref{m} holds. Then there exists a positive minimizer $u \in C^{1,\alpha}(\Omega) \cap L^{\infty}(\Omega)$ of \eqref{var.beta} which is a weak solution to \eqref{bvp.var.beta1} in $\Omega$ with $\ell=\ell_{1}(\beta, \Omega)$. Moreover $\ell_1(\beta, \Omega)$ is positive and it is simple, that is the relative eigenfunction $u$ is unique up to a multiplicative constant. 
\end{prop}
\begin{proof}
Let $u_k\in W^{1,p}(\Omega)$ be a minimizing sequence of \eqref{var.beta} such that $\Vert u_k\Vert_{L^p(\Omega)}=1$. Then, being $u_k$ bounded in $W^{1,p}(\Omega)$   there exists a subsequence, still denoted by $u_k$ and a function $u\in W^{1,p}(\Omega)$ with $\Vert u\Vert_{L^p(\Omega)}=1$, such that $u_k\rightarrow u$ strongly in $L^p(\Omega)$ and $\D u_k\rightharpoonup \D u$ weakly in $L^p(\Omega)$. Then  $u_k$  converges to $u$ in $L^p(\de\Omega)$ and then almost everywhere on $\de \Omega$ to $u$. Then by the weak lower semicontinuity and Fatou's lemma we get 
\begin{equation*}
\ell_1(\beta, \Omega)=\lim_{k\rightarrow +\infty}J[\beta,u_k]\geq J[\beta,  u],
\end{equation*}
then $\beta(\cdot)\vert u\vert^p\in L^1(\de\Omega)$ and $u$ is an eigenfunction related to $\ell_{1}(\beta, \Omega)$ by definition.    Moreover $u \in L^{\infty}(\Omega)$. To see that, we  can argue exactly as in \cite{pota} in order to get that $u \in L^{\infty}(\Omega)$.

Now the $L^{\infty}$-estimate, the hypothesis \eqref{strong} and the properties of $F$ allow to apply standard regularity results (see \cite{db83}, \cite{tk84}), in order to obtain that $u\in C^{1,\alpha}(\Omega)$. 
%
%As matter of fact, as observed in \cite{bd10} it is possible to follow the argument in \cite[pages 466-467]{ladyz} to get the continuity of $u$ up to $\de \Omega$.

In order to prove that $\ell_{1}(\beta, \Omega) >0 $, we procede  by contraddiction supposing that there exists $\beta_{o}$ which verifies \eqref{m} and  such that $\ell_1(\beta_{o}, \Omega)=0$. Then there exists $u_{\beta_{o}}\in C^{1,\alpha}(\Omega)\cap L^{\infty}(\overline{\Omega})$ such that $u_{\beta_{o}} \ge 0$,  $\norm{u_{\beta_{o}}}_{L^{p}(\Omega)}=1$ and
\begin{equation*}
0=\ell_1(\beta_{o},\Omega)=\ds\int_{\Omega}F^p(\D u_{\beta_{o}})dx+\ds\int_{\de\Omega}\beta_{o} \,u_{\beta_{o}}^p\,F(\nu)d\haus.
\end{equation*}
Then $u_{\beta_{o}}$ has to be constant in $\overline{\Omega}$ and then  $u_{\beta_{o}}^p\ds\int_{\de\Omega}\beta_{o} F(\nu)=u_{\beta_{o}}^p m= 0$. Being $m >0$, then  $u_{\beta_{o}}=0 \mbox{ in }\ \overline{\Omega}$, and this is not true. Hence $\ell_1( \beta_{o},\Omega)>0$.\par 
Finally to prove  the semplicity of the eigenfunctions we can procede exactly as in \cite{pota}. For completeness  we recall the main steps. 
Let $u,w$ be  positive minimizers of  the functional $J$ defined in \ref{var.beta} such that $\|u\|_{p}=\|w\|_{p}=1$, and let us consider  the function $\eta_{t}=(t u^{p}+(1-t)w^{p})^{1/p}$, with $t\in[0,1]$.  Obviously, $\|\eta_{t}\|_{p}=1$. Clearly it holds:
\begin{equation}
\label{absurd}
J[\beta,u]=\ell_1( \beta,\Omega) =J[\beta,w]. 
\end{equation}
In order to compute $J[\beta,\eta_{t}]$ we observe that 
by using the homogeneity and the convexity of $F$ it is not hard to prove that (see for instance \cite{pota} for the precise computation)
\begin{equation}
\label{conv}
F^{p}(\D\eta_{t})  \le t F^{p}(\D v)+(1-t)F^{p}(\D w).
\end{equation}
%arguments: let $v,w$ two positive minimizers of \ref{var.beta} in $\Omega$ such that $\|v\|_{p}=\|w\|_{p}=1$, and consider $\eta_{t}=(t v^{p}+(1-t)w^{p})^{1/p}$, with $t\in[0,1]$.  Obviously, $\|\eta_{t}\|_{p}=1$. Moreover, using the homogeneity and the convexity of $H$ we get that
%\begin{equation}
%	\label{convx}
%	\begin{array}{rl}
%		[H(D\eta_{t})]^{p} & = \eta_{t}^{p} \left[H\left( t 
%		\left(\dfrac{v}{\eta_{t}} \right)^{p}
%		\dfrac{Dv}{v} + (1-t)\left(\dfrac{w}{\eta_{t}} \right)^{p}		\dfrac{Dw}{w} \right)\right]^{p} \\[.4cm]
%		&=\eta_{t}^{p} \left[H\left( s(x) 
%		\dfrac{Dv}{v} + (1-s(x)) \dfrac{Dw}{w} \right)\right]^{p} \\[.4cm]
%		&\le \eta_{t}^{p} \left[ s(x) H\left(
%		\dfrac{Dv}{v}\right) + (1-s(x)) H \left( \dfrac{Dw}{w}\right) \right]^{p} \\[.4cm]
%		&\le tv^{p} \left[H\left(
%		\dfrac{Dv}{v}\right)\right]^{p} + (1-t)w^{p} \left[H \left( \dfrac{Dw}{w}\right)\right]^{p} \\[.4cm]
%		 &= t [H(D v)]^{p}+(1-t)[H(Dw)]^{p}.
%	\end{array}
%\end{equation}
%
Hence recalling \eqref{absurd}, we obtain 
\[
J[\beta,\eta_{t}] \le t J[\beta,u]+ (1-t)J[\beta,w] =\ell_{1}(\beta,\Omega),
\]
and then $\eta_{t}$ is a minimizer for $J$. This implies that the  equality holds in \eqref{conv}, and as showed in \cite{pota}, this implies that $u=w  $ that is the uniqueness.
% The equality between the third and the fourth row of \eqref{convx} holds if and only if $H(Dv/v)=H(Dw/w)$. Hence, the strict convexity of the level sets of $H$ guarantees from the equalities in \eqref{convx} that $D v/v= Dw/w$ in $\Omega$, that is $v/w$ is constant. The norm constraint on $v$ and $w$ implies the uniqueness, and this concludes the proof.
\end{proof}
The following result characterizes the first eigenfunctions. 
\begin{prop}\label{carac_eig_betavar}
Let $\beta\in L^1(\de\Omega)$, $\beta\geq 0$ be such that \eqref{m} holds. Let  $\eta >0$ and  $v\in W^{1,p}(\Omega)$, $v \not \equiv 0$ and   $v \ge0$ in $\Omega$ such that
\begin{equation*}\label{prob.gen}
\begin{cases}
-\Qp v=\eta v^{p-1} &\mbox{in}\ \Omega\\[.2cm]
F^{p-1}(\D v)F_\xi(Dv)\cdot\nu + \beta F(\nu) v^{p-1}=0 &\mbox{on}\ \de\Omega
\end{cases}
\end{equation*}
in the sense of Definition \ref{aut}. Then $v$ is a first eigenfunction of \eqref{bvp.var.beta}, and $\eta=\ell_1(\beta,\Omega)$.
\begin{proof}
Let $u \in W^{1,p}(\Omega)$ be a positive  eigenfunction related to $\ell_{1}(\beta, \Omega)$. 
Choosing $u^{p}/(v+\eps)^{p-1}$, with $\eps>0$, as test function in the Definition \ref{aut} for the solution $v$, and arguing exactly as in \cite{pota}, we get the  claim.
\end{proof}
\begin{rem}
We observe that Propositions \ref{eig_betavar} and \ref{carac_eig_betavar} generalize the results proved respectively in \cite{DpGK} for the Euclidean norm and in \cite{pota} when $\beta(x)=\beta$ is a positive constant.   
\end{rem}

\end{prop}

\begin{thm}
\label{th:lbdpr2}
Let $\beta\in L^1(\de\Omega)$, $\beta\geq 0$ and such that \eqref{m} holds. The following properties hold for $\ell_1(\beta,\Omega)$
 \begin{enumerate}
 % \item $\sup_{\vert\Omega\vert =m}\lambda_1(\Omega)=+\infty$;
	\item[(i)] $\forall t>0,\ \ell_1(\beta(\frac{x}{t}), t\Omega)=t^{-p}\ell_1(t^{p-1}\beta(y), \Omega), \quad x \in \de(t \Omega), y\in \de\Omega$ ;
	%\item[(ii)]  $\ell_1(\beta, \Omega)\leq\lambda_D(\Omega)$;
	\item[(ii)]  $\ell_1(\beta, \Omega)\leq \frac{m}{\vert\Omega\vert}$;
	\item[(iii)] $a^{p}\ell_{\mathcal E}(a^{1-p}\beta, \Omega)\le\ell_1(\beta, \Omega)\leq b^{p}\ell_{\mathcal E}(b^{1-p}\beta, \Omega)$,
	
	where $a,b$ are defined in \eqref{eq:lin},\eqref{upb} and $\ell_{\mathcal E}(a^{1-p}\beta, \Omega)$, $\ell_{\mathcal E}(b^{1-p}\beta, \Omega)$ are the first Robin eigenvalue for the Euclidean $p$-Laplacian corresponding respectively to the function $a^{1-p}\beta$ and $b^{1-p}\beta$;
	\item[(iv)]  If $\beta(x)\ge \bar\beta>0,\ \mbox{for almost }\ x\in\de\Omega$, then 
	$$\sup_{\vert\Omega\vert =k}\ell_1(\beta , \Omega)=+\infty$$
 \end{enumerate}
\end{thm}
\begin{proof}
 By the homogeneity of $F$, we have: 

\begin{multline*}
\ell_1\left(\beta\left(\frac{x}{t}\right), t\Omega\right)=\min_{\substack{\varphi\in W^{1,p}(t\Omega)\\ \varphi\neq 0}}
\frac{\ds\int_{t\Omega} F^{p}(\D \varphi(x))dx +\int_{\de(t\Omega)}\beta\left(\frac{x}{t}\right)|\varphi(x)|^{p}F(\nu(x))d\haus(x)}{\ds\int_{t\Omega}|\varphi(x)|^{p}dx}\\[.15cm]
=\min_{\substack{v\in W^{1,p}(\Omega)\\ v\neq 0}}
\frac{t^{-p}\ds\int_{t\Omega} F^{p}\left(\D v\left(\frac{x}{t}\right)\right)dy +\int_{\de(t\Omega)}\beta\left(\frac{x}{t}\right)|v\left(\frac{x}{t}\right)|^{p}F\left(\nu\left(\frac{x}{t}\right)\right)d\haus(y)}{t^{n}\ds\int_{\Omega}|v(y)|^{p}dy}=\\[.15cm]
=\min_{\substack{v\in W^{1,p}(\Omega)\\ v\neq 0}}
\frac{t^{n-p}\ds\int_{\Omega} F^{p}(\D v(y))dy +t^{n-1}\int_{\de\Omega}\beta(y)|v(y)|^{p}F(\nu(y))d\haus(y)}{t^{n}\ds\int_{\Omega}|v(y)|^{p}dy}=t^{-p}\ell_1( t^{p-1}\beta(y), \Omega).
\end{multline*}

In order to obtain the second property, 
it is sufficient  to  consider a non-zero constant as test function in \eqref{var.beta}. 

Now we prove the  inequality in the right-hand side in  $(iii)$. The proof of the other inequality is similar.  By using \eqref{var.beta} and \eqref{upb}, we obtaine that
\begin{multline*}
\ell_1( \beta, \Omega)=\inf_{\substack{v\in W^{1,p}(\Omega)\\ v\not\equiv 0}}\frac{\ds\int_{\Omega} F^p(\D v) dx+\ds\int_{\de\Omega}\beta(x)\vert v\vert^p F(\nu)d\haus}{\ds\int_{\Omega}\vert v\vert^p dx}\le\\\inf_{\substack{v\in W^{1,p}(\Omega)\\ v\not\equiv 0}}b^{p}\frac{\ds\int_{\Omega} |\D v|^{p} dx+b^{1-p}\ds\int_{\de\Omega}\beta(x)\vert v\vert^p d\haus}{\ds\int_{\Omega}\vert v\vert^p dx}=\\b^{p}\inf_{\substack{v\in W^{1,p}(\Omega)\\ v\not\equiv 0}}\frac{\ds\int_{\Omega} |\D v|^{p} dx+\ds\int_{\de\Omega}b^{1-p}\beta(x)\vert v\vert^p d\haus}{\ds\int_{\Omega}\vert v\vert^p dx}=b^{p}\ell_{\mathcal E}(b^{1-p}\beta, \Omega),\end{multline*}
where last equality follows, by definition of $\ell_{\mathcal E}(b^{1-p}\beta, \Omega)$.

Finally we give the proof of $(iv)$. Clearly $\ell_1(\beta,\Omega)\ge\ell_1(\bar\beta,\Omega)$, then by \cite[Proposition 3.1]{pota}, we know that  
\begin{equation}
\label{st_pota}
\ell_1( \bar\beta , \Omega ) \ge \left( \frac{p-1}{p} \right)^p \frac{ \bar\beta}{ R_{ F}\left( 1+{\bar\beta}^\frac{1}{p-1} R_{F}\right)}
\end{equation}
where $R_{F}$ is the anisotropic inradius of the subset $\Omega$ . The claim follows constructing a sequence of convex  sets $\Omega_k$  with $\abs{\Omega_k}=1$ and such that $R_{F}(\Omega_k) \rightarrow 0$, for $k \to \infty$.
Let $k >0$, proceeding as in \cite{DpDbG,dpgg},  it is possible to consider the $n-$rectangles  $\Omega_k=   \left] -\frac{1}{2k}, \frac{1}{2k}\right[\times \left]-\frac{k^{\frac{1}{n-1}}}{2},\frac{k^{\frac{1}{n-1}}}{2}\right[^{n-1}$ and suppose that $R_F(\Omega_k) = \displaystyle\frac{1}{2k} F^o(e_{1})$. Then we obtain
\begin{equation*}
\ell_1(\bar\beta, \Omega_k) \ge \left( \frac{p-1}{p} \right)^p \frac{ 4k^2\bar\beta}{ F^o(e_{1})\left( 2k+{\bar\beta}^\frac{1}{p-1} F^o(e_{1})\right)}  \to +\infty\mbox{ for } k \to \infty. 
\end{equation*}
\end{proof} 
\section{The anisotropic radial case}
In this section we recall some properties of the first eigenvalue of $\Qp$  with Robin boundary condition when $\Omega$ is a Wulff shape. We suppose that $\beta=\bar\beta $ is a positive constant then  we consider 
\begin{equation}
\label{minrad}
\ell_{1}(\bar\beta, \mathcal W_{R})=\min_{\substack{v\in W^{1,p}(\mathcal W_{R}) \\ v\not\equiv 0}} J( \bar\beta,v) =\min_{\substack{v\in W^{1,p}(\mathcal W_{R}) \\ v\not\equiv 0}}\dfrac{\ds\int_{\mathcal W_{R}} [F(\D v)]^{p}dx +\bar\beta\int_{\de \mathcal W_{R}}|v|^{p}F(\nu)d\haus}{\ds\int_{\mathcal W_{R}} |v|^{p}dx},
\end{equation}
where $\mathcal W_{R}=R\,\mathcal W=\{x\colon F^{o}(x)<R\}$, with $R>0$, and $\mathcal W$ is the Wulff shape centered at the origin.

By Proposition \ref{eig_betavar}, the minimizers of \eqref{minrad} solve the following problem:
\begin{equation}
  \label{eq:3}
  \left\{
   \begin{array}{ll}
     -\Qp v=\ell_1(\bar\beta,\mathcal W_{R}) 
      |v|^{p-2}v 
      &\text{in } 
      \mathcal W_{R},\\[.2cm] 
      % H(Du)^{p-1}H_\xi(D u)
      F^{p-1}(\D v)F_\xi(Dv)\cdot\nu +\bar\beta F(\nu) |v|^{p-2}v=0
      &\text{on } \de\mathcal W_{R}.
    \end{array}
  \right.
\end{equation}
In \cite{pota,dan,bd10} the authors prove the following result
\begin{thm}
\label{teorad}
Let $v_{p}\in C^{1,\alpha}(\mathcal W_R)\cap C(\overline{\mathcal W_R})$ be a positive solution to problem \eqref{eq:3}. Then $v_{p}(x)=\varrho_{p}(F^{o}(x))$, with  $x\in  \overline{\mathcal W}_{R}$, where $\varrho_{p}(r)$, $r\in [0,R]$, is a 
 decreasing function  such that  $\varrho_{p}\in C^{\infty}(0,R)\cap C^{1}([0,R])$ and it verifies 
%	\begin{equation}
%		\label{eqrad}
%		\left\{
%	\begin{array}{ll}
%-(p-1)(-\varrho'_{p}(r))^{p-2} \varrho_{p}''(r)+\dfrac{n-1}{r}(-\varrho'_{p}(r))^{p-1}	
%	=\lambda_{1,\mathcal E}(B_{R})\varrho_{p}(r)^{p-1},&r\in]0,R[,\\[
\begin{equation}\label{eqrad}
\left\{
		\begin{array}{l}
		-(p-1)(-\varrho'_{p}(r))^{p-2} \varrho_{p}''(r)+\dfrac{n-1}{r}(-\varrho'_{p}(r))^{p-1}	
	=\ell_1(\bar\beta,\mathcal W_{R}) \varrho_{p}(r)^{p-1},\quad r\in]0,R[,
						\\[.15cm]
			\varrho_{p}'(0)=0,\\[.15cm]
			-(-\varrho_{p}'(R))^{p-1} + \bar\beta (\varrho_{p}(R)) ^{p-1}=0.
		\end{array} 
	\right.
\end{equation}
\end{thm}

\begin{rem}
We observe that   the first eigenvalue in the Wulff $\mathcal W_{R}=\{F^{o}(x)<R\}$ is the same for any norm $F$. In particular it coincides with  the first Robin eigenvalue in the Euclidean ball $B_{R}$ for the $p$-Laplace operator.  Finally we emphasize   that in this case the eigenfunctions have more regularity because $\bar \beta$ is a positive constant. 
\end{rem}
Theorem \ref{teorad}, as in \cite{bd10,pota}, suggests to consider , for every $x\in\W_R$,   the following function 
\begin{equation}
\label{betar}
f(r_{x})= \frac{(-\rho_p^\prime(r_x))^{p-1}}{(\rho_p(r_x))^{p-1}} = \frac{\left[F\big(\D v_{p}(x)\big)\right]^{p-1}}{{v_{p}(x)}^{p-1}} = \frac{\left[F\big(\D v_p(x)\big)\right]^{p-1}F_\xi\big(Dv_p(x)\big)\cdot\nu}{{v_{p}(x)}^{p-1}F(\nu)},
\end{equation}
where
\begin{equation*}
\label{rx}
 r_x=F^o(x),  \quad 0 \le r_{x} \le R.
\end{equation*} 
Let us observe that $f$ is nonnegative, $f(0)=0$ and 
\[
f(R)=\frac{(-\rho_p^\prime(R))^{p-1}}{(\rho_p(R))^{p-1}} = \frac{\left[F\big(\D v_{p}(x)\big)\right]^{p-1}}{{v_{p}(x)}^{p-1}} = \frac{\left[F\big(\D v_p(x)\big)\right]^{p-1}F_\xi\big(\D v_p(x)\big)\cdot\nu}{{v_{p}(x)}^{p-1}F(\nu)}=\bar\beta
\]
 
%Next two lemmata will be useful in the proof of the main result. Their proofs are analogous to the ones obtained in \cite{bd10}. For the sake of completeness, we write them in details.

The following result proved in the Euclidean case in \cite{bd10} and in \cite{pota} in the anisotropic case,  states that the first Robin eigenvalue is monotone decreasing with respect the set inclusion in the class of Wulff shapes. 
\begin{lem}
\label{monotonia1}
The function $r \to \ell_{1}(\bar\beta,\mathcal W_{r})$  is strictly decreasing in $]0,\infty[$.
%If $0<r<s$, then $\ell_{1}(\bar\beta,\mathcal W_{r})>\ell_{1}(\bar\beta,\mathcal W_{s})$. 
\end{lem}

In \cite{bd10} and \cite{pota} the authors prove also the following monotonicity  property for the function $f$ defined in \eqref{betar}.
\begin{lem}
\label{lemmabeta}
Let $f$ be the function defined in \eqref{betar}. Then $f(r)$ is strictly increasing in $[0,R]$.
\end{lem}
In the next result we prove a convex property  for the function $f$.
\begin{thm}
\label{tconv}
Let $f$ be the function defined in \eqref{betar}. Then the function
\begin{equation*}
g(z)=f(z^{\frac 1 n}) z^{1-\frac{1}{n}} , \qquad  0 \le z \le R^{n}, 
\end{equation*}
is convex with respect to $z$.
\end{thm}
\begin{proof}
We first observe that by  \eqref{eqrad} it holds
\begin{multline}
\label{eqf}
f'(r)=\frac{d}{dr}\left(\frac{-\rho_p^\prime(r)}{\rho_p(r)}  \right)^{p-1}=(p-1)\displaystyle f^{\frac{p-2}{p-1}} \left( \frac{-\rho_p^{\prime\prime}}{\rho_p}+\left(\frac{\rho_p^{\prime}}{\rho_p}\right)^{2}\right)\\=\ell_1(\bar\beta,\mathcal W_{R})  - \frac{(n-1)}{r}f+(p-1)\displaystyle f^{\frac{p}{p-1}} \quad \forall r \in ]0,R[.
\end{multline}
 Then 
\begin{multline*}
g'(z)=\frac{1}{n} f'(z^{\frac 1 n})+\frac{(n-1)}{n} \frac{f(z^{\frac 1 n})}{z^{\frac 1 n}} \\ =\frac{1}{n}\left(\ell_1(\bar\beta,\mathcal W_{R})  - \frac{(n-1)}{z^{\frac 1 n}}f(z^{\frac 1 n})+(p-1)\displaystyle f^{\frac{p}{p-1}}(z^{\frac 1 n})\right)+ \frac{(n-1)}{n} \frac{f(z^{\frac 1 n})}{z^{\frac 1 n}}\\=\frac{\ell_1(\bar\beta,\mathcal W_{R})}{n}+\frac{(p-1)}{n}\displaystyle f^{\frac{p}{p-1}}(z^{\frac 1 n}),
\end{multline*}
 which is increasing and this implies the thesis.
\end{proof}
Finally the following comparison result for $f$ holds
\begin{thm}
\label{confr} 
Let $f$ be the function defined in \eqref{betar}. Then  there exists a positive constant $C=C(R)$ such that 
\begin{equation*}
\label{comp}
f(r) \le r \, C(R),  \,\,\, \text{ for }\,\, 0 \le r \le R
\end{equation*}
\end{thm}
\begin{proof}
By  \eqref{eqf} and by Lemma \ref{lemmabeta} we obtain that $f$ verifies the following equation
\begin{equation}
\label{app1}
f'(r)=\ell_1(\bar\beta,\mathcal W_{R})  - \frac{(n-1)}{r}f(r)+(p-1)\displaystyle f^{\frac{p}{p-1}}(r) \le  C(R) - \frac{(n-1)}{r}f(r) 
\end{equation}
where 
\begin{equation}
\label{costc}
C(R)=\ell_1(\bar\beta,\mathcal W_{R})+(p-1)\displaystyle f(R)^{\frac{p}{p-1}}=\ell_1(\bar\beta,\mathcal W_{R})+(p-1)\displaystyle \bar\beta^{\frac{p}{p-1}}
\end{equation}
Then by \eqref{app1} multiplying both sides by $r^{n-1}$ we get
\begin{equation*}
f'(r)r^{n-1}+(n-1)r^{n-2}f(r)\le C(R)r^{n-1},
\end{equation*}
and 
\[
\displaystyle\frac{d}{dr} \left( r^{n-1}f(r)\right)\le C(R)r^{n-1}
\]
Then the claim follows integrating both sides between $0$ and $r$.
\end{proof}
\begin{rem}
\label{rwii}
The  results contained in  Lemma \ref{lemmabeta} and Theorem \ref{tconv} ensures that 
$f(r)$ is an admissible weight for the  weighted anisotropic isoperimetric inequality  quoted in \eqref{wii}\end{rem}
%%%%%%%%%%%%%%%%%%%%%%%%%%%%%%%%%%%%%%%%%%%%%%%%%%%%%%%%%%%%%%%%%%%%%
\section{ A monotonicity property for  $\ell_{1}(\bar\beta;\Omega)$} 

In this section we assume that $\beta=\bar\beta $ is a positive constant. The first Robin eigenvalue $\ell_1(\bar\beta, \Omega)$ has  not, in general,  a monotonicity  property with respect the set inclusion.  For instance  in  \cite{dd}  in the Euclidean case, for the Laplace operator, the authors give a counterexample. More precisely, they    construct  a suitable sequence of sets $\Omega_k$ such that  $P_{\mathcal E}(\Omega_k) \to \infty$, $B_1(0) \subset \Omega_k \subset B_{1+\eps}(0)$ which verify
\[
\ell_{\mathcal E}(\bar \beta, \Omega_k) > \ell_{\mathcal E}(\bar \beta, B_1(0)) >\ell_{\mathcal E}(\bar \beta, B_2(0)). 
\]
Here $B_r(x_o)$ denotes the Euclidean ball with radius $r$ and centered at the pint $x_o$ and $\lambda_{\mathcal E}(B_{1+\eps}(0))$ is the  Euclidean first  Dirichlet  eigenvalue of the Laplacian of the ball  $B_{1+\eps}(0)$.
 In what follows we prove a monotonicity  type property for the first Robin eigenvalue of the operator $\Qp$ with respect the set inclusion. In the Euclidean case for the Laplace operator  we refer the reader for instance to \cite{gs}.  
\begin{thm}
\label{mono1}
Let $\Omega \subset \R^{n}$ be a bounded open set with $C^{1,\alpha}$ boundary,  $\alpha \in ]0,1[$.
Let $\mathcal{W}_R$ be a Wulff shape such that  $\Omega\subset\mathcal{W}_R$ and $\bar \beta$ a positive constant. Then
 \[
    \ell_1(\bar\beta, \mathcal{W}_R)\leq\ell_1(\bar\beta,\Omega).
 \]
\end{thm}
 \begin{proof}
Let $v_p$ be the positive eigenfunction associated to $\ell_1( \bar\beta,\mathcal{W}_R)$ and let $\Omega$ be a subset of $\mathcal{W}_R$.  

Then for every $x\in\de\Omega$, we can consider $f(r_{x})$ as in \eqref{betar} in order to get that the following Robin boundary condition on $\de \Omega$  holds
\begin{equation}\label{bou1}
\left[F \big(\D v_p(x)\big)\right]^{p-1}F_\xi\big(\D v_p(x)\big)\cdot\nu+f(r_x){v_p(x)}^{p-1}F(\nu)=0.
\end{equation}
Having in mind that $\Omega\subset\W_R$ and using \eqref{bou1}, we have that $v_p$ solves the following problem
\begin{equation}
\label{pvp}
 \begin{cases}
  -\Qp v_p=\ell_{1}(\bar\beta,\mathcal{W}_R)v_p^{p-1} & \mbox{in}\ \Omega\\[.2cm]
	\left[F \big(\D v_p\big)\right]^{p-1}F_\xi\big(\D v_p\big)\cdot\nu+f(r_x){v_p}^{p-1}F(\nu)=0 & \mbox{on}\ \partial\Omega
 \end{cases}
\end{equation}     
Using \eqref{pvp} and Lemma \ref{lemmabeta}
\begin{equation*}
\begin{split}
\ell_1(\bar\beta,\mathcal{W}_R) & = \frac{\displaystyle\int_\Omega [F(\D v_p)]^pdx+\displaystyle\int_{\partial\Omega}f(r_x)\vert v_p\vert^p F(\nu)d\haus}{\displaystyle\int_\Omega \vert v_p\vert^pdx}\\
  & =\inf_{u\in W^{1,p}(\Omega)\setminus\{0\}}\frac{\displaystyle\int_\Omega [F(\D u)]^pdx+\displaystyle\int_{\partial\Omega}f(r_x)\vert u\vert^p F(\nu)d\haus}{\displaystyle\int_\Omega \vert u\vert^pdx}\\
  & \leq\inf_{u\in W^{1,p}(\Omega)\setminus\{0\}}\frac{\displaystyle\int_\Omega [F(\D u)]^pdx+\displaystyle\int_{\partial\Omega}\bar\beta\vert u\vert^p F(\nu)d\haus}{\displaystyle\int_\Omega \vert u\vert^pdx}\\[0.2cm] 
	& = \ell_1(\bar\beta,\Omega)
\end{split}
\end{equation*}
 \end{proof}
When $\Omega$ contains a Wulff shape we have the following result
\begin{thm}
\label{mono2}
Let $\Omega \subset \R^{n}$ be a bounded and convex open set with $C^{1,\alpha}$ boundary,  $\alpha \in ]0,1[$. Let $\mathcal{W}_R$ be a Wulff shape  such that $\mathcal{W}_R\subset\Omega$, then
 \[
   \ell_1(\bar\beta, \Omega)\leq\ell_1( \bar\beta,\mathcal{W}_R).
 \]
\begin{proof}
 First of all, we take the positive eigenfunction $v_p$ associated to $\ell_1(\bar\beta,\W_R)$.  By Theorem 4.1 $v_p(x)=\varrho_{p}(F^o(x))$, and by \eqref{eqrad}  we  can extend $\varrho_{p}$ up to $+ \infty$ and then  $v_p$ in   $\R^n$.   
Let us consider the super-level set 
\begin{equation*}
\W_{+}=\{x \in \R^n \colon v_p(x)>0\}.
\end{equation*}
By the property of $v_{p}$ $\W_{+}$ is a Wulff shape and clearly $\W_R \subset \W_+$. 

 Moreover, $v_p$ solves the following equation
 \begin{equation*}\label{eqVP}
   -\Qp v_p= \ell_1(\bar\beta, \W_R )v_p^{p-1}\ \mbox{in}\ \W_{+}.
 \end{equation*}
To prove the Theorem we consider the  set $\tilde{\Omega}=\Omega\cap\W_{+}$. Being $\Omega $ convex and due to the radially decreasing of the eigenfunction,  three possible cases can occour.\\ 

\textit{Case 1:}  $\de\tilde{\Omega}=\de\Omega$. Then in this case $\W_R \subset \Omega\subset\W_{+}$ and $\tilde{\Omega}=\Omega$. Then for $x \in \de \Omega$ we  put $r_x=F^o(x)$ and  we can compute 
$$f(r_x)=\frac{(-\rho_p^\prime(r_x))^{p-1}}{(\rho_p(r_x))^{p-1}}.$$
Then   arguing as in the proof of Theorem \ref{mono1} and recalling that by Lemma \ref{lemmabeta},  $f(r_x)\ge \bar \beta$, for any $x \in \de \Omega$  we get
 \begin{equation*}
  \begin{split}
	 \ell_1(\bar\beta,\mathcal{W}_R) & = \frac{\displaystyle\int_\Omega [F(\D v_p)]^pdx+\displaystyle\int_{\partial\Omega}f(r_x)\vert v_p\vert^p F(\nu)d\sigma}{\displaystyle\int_\Omega \vert v_p\vert^pdx}\\
  & =\inf_{u\in W^{1,p}(\Omega)\setminus\{0\}}\frac{\displaystyle\int_\Omega [F(\D u)]^pdx+\displaystyle\int_{\partial\Omega}f(r_x)\vert u\vert^p F(\nu)d\sigma}{\displaystyle\int_\Omega \vert u\vert^pdx}\\
  & \geq\inf_{u\in W^{1,p}(\Omega)\setminus\{0\}}\frac{\displaystyle\int_\Omega [F (\D u)]^pdx+\displaystyle\int_{\partial\Omega}\bar\beta\vert u\vert^p F(\nu)d\sigma}{\displaystyle\int_\Omega \vert u\vert^pdx}\\[0.2cm] 
	& = \ell_1(\bar\beta,\Omega)
	\end{split}
 \end{equation*}
and the first case is proved.\\

\textit{Case 2:} $\de\tilde{\Omega}\cap\de\Omega\neq\emptyset$ and $\de\tilde{\Omega}\cap\de\Omega\neq\de\Omega$. Then $\de\tilde{\Omega}\cap\W_{+}\neq\emptyset$. Moreover, on $\de\tilde{\Omega}\cap\de\Omega$ the eigenfunction $v_p$ is positive, while on $\de\tilde{\Omega}\cap\de\W_{+}$ it is equal to zero. In particular, for every $x\in\de\tilde{\Omega}\cap\de\Omega$ we still have that $f(r_x) \ge \bar\beta$ as in the \textit{Case 1}. We define the following test function $\varphi\in W^{1,p}(\Omega)$ 
 \begin{equation*}
\varphi(x)=
  \begin{cases}
	v_p(x) & \mbox{in}\ \tilde{\Omega}\\
	0 & \mbox{in}\ \Omega\setminus\tilde{\Omega}.
	\end{cases}
 \end{equation*}
 Then
\begin{equation*}
  \begin{split}
	 \ell_1( \bar\beta,\mathcal{W}_R) & = \frac{\displaystyle\int_{\tilde{\Omega}} [F(\D v_p)]^pdx+\displaystyle\int_{\partial\tilde{\Omega}\cap \de \Omega}f(r_x)v_p^p F(\nu)d\haus}{\displaystyle\int_{\tilde{\Omega}} v_p^pdx}\\
  & =\frac{\displaystyle\int_\Omega [F(\D \varphi)]^pdx+\displaystyle\int_{\de\tilde{\Omega}\cap\partial\Omega}f(r_x)\varphi^p F(\nu)d\haus}{\displaystyle\int_\Omega v^pdx}\\
	& \geq\frac{\displaystyle\int_\Omega [F(\D \varphi)]^pdx+\displaystyle\int_{\de\tilde{\Omega}\cap\partial\Omega}\bar\beta \varphi^p F(\nu)d\haus}{\displaystyle\int_\Omega v^pdx}\\
  &=\frac{\displaystyle\int_\Omega [F (\D v)]^pdx+\displaystyle\int_{\partial\Omega}\bar\beta \varphi^p F(\nu)d\haus}{\displaystyle\int_\Omega  v^pdx}
	\end{split}
	\end{equation*}
	\begin{equation*}
	\begin{split}
	\, & \geq\inf_{u\in W^{1,p}(\Omega)\setminus\{0\}}\frac{\displaystyle\int_\Omega [F (\D u)]^pdx+\displaystyle\int_{\partial\Omega}\bar\beta\vert u\vert^p F(\nu)d\haus}{\displaystyle\int_\Omega \vert u\vert^pdx}\\[0.2cm]
	& = \ell_1(\bar\beta, \Omega)
	\end{split}
 \end{equation*}
and the second case is proved.\\

\textit{Case 3:} $\de\tilde{\Omega}\cap\de\Omega =\emptyset$. Then $\tilde{\Omega}=\W_{+}\subset\Omega$. Using the monotonicity result in Lemma \ref{monotonia1} we obtain that $\ell_1(\bar\beta,\W_R)\geq\ell_1(\bar\beta,\W_{+})$. Denoting with $v_p^{(1)}$ the eigenfunction associated to $ \ell_1(\bar\beta, \W_{+}) $  and defining $\tilde{\Omega}^{(1)}=\Omega\cap\{v_p^{(1)}(x)>0\}$ and repeating the division in three possible cases, after a finite number of steps we could be either in \textit{Case 1} or in \textit{Case 2}.
\end{proof}
\end{thm}
By Theorems \ref{mono1} and \ref{mono2} we get the following monotonicity property for $\ell_{1}$ for constant $\beta$.
\begin{cor}
Let $\Omega_{1}, \Omega_{2} \subset \R^{n}$ be as in \eqref{dom_cond} and  convex. Let $\mathcal{W}_R$ be a Wulff shape  such that $\Omega_{1} \subset\mathcal{W}_R\subset\Omega_{2}$. Then $\ell_{1}(\bar\beta, \Omega_{2}) \le \ell_{1}(\bar\beta, \Omega_{1}) $.
\end{cor}

%%%%%%%%%%%%%%%%%%%%%%%%%%%%%%%%%%%%%%%%%%%%%%%%%%%%%%%%%%%%%%%%%%%%%
\section{A representation formula for $\ell_1(\beta, \Omega )$}
In this section we prove a level set representation formula for the first eigenvalue $\ell_{1}( \beta , \Omega )$ of the following problem
\begin{equation}\label{bvp.var.beta}
\begin{cases}
-\Qp v=\ell \vert v\vert^{p-2}v &\mbox{in}\ \Omega\\[.2cm]
F^{p-1}(\D v)F_\xi(\D v)\cdot\nu + \beta F(\nu)\vert v\vert^{p-2}v=0 &\mbox{on}\ \de\Omega.
\end{cases}
\end{equation} 
When $\beta=\bar\beta$ is a nonnegative constant a similar result can be found in  \cite{bd10} in the Euclidean case and in \cite{pota} for the anisotropic case.  Our aim is to extend the known results assuming that $\beta$ is in general a function defined on $\de \Omega$. In the next we will use the following notation. Let $\utilp$ be the first positive eigenfunction such that $\max \utilp=1$. Then, for $t\in[0,1]$,
\begin{equation*}
	\begin{array}{l}
		U_{t} =\{x\in \Omega\colon \utilp>t\},\\[.1cm]
		S_{t} =\{x\in \Omega\colon \utilp=t\},\\[.1cm]
		\Gamma_{t}=\{x\in \de\Omega\colon \utilp>t\}.
	\end{array}
\end{equation*}
\begin{thm}\label{ug_beta_var}
Let $\Omega \subset \R^{n}$ be a bounded open set with $C^{1,\alpha}$ boundary and let $\alpha \in ]0,1[$.
Let $\beta$ be a function belonging to $L^1(\de\Omega)$, $\beta\geq 0$ and such that \eqref{m} holds. Let  $\utilp \in C^{1,\alpha}(\Omega) \cap L^{\infty}(\Omega)$ be a positive minimizer of $\eqref{var.beta}$ with $\norm{\utilp}_{\infty} = 1$. Then for a.e. $t \in ]0,1[ $   the following representation formula holds
\begin{equation}\label{rep.ell1}
\ell_1(\beta , \Omega) = \mathcal{F}_{\Omega}\left(U_t , \frac{\left[F(\D \utilp)\right]^{p-1}}{{\utilp}^{p-1}}\right),
\end{equation}
where $\mathcal{F}_{\Omega}$ is defined as 
\begin{equation}\label{F.betavar}
\mathcal{F}_{\Omega}(U_{t},\varphi)=\frac{1}{|U_{t}|}\left(
-(p-1)\int_{U_{t}} \varphi^{p'} \,dx + \int_{ S_{t}} 
\varphi F(\nu)\,d\sigma +
 \int_{ \Gamma_{t}}\beta F(\nu)\,d\sigma\right).
\end{equation}
\end{thm}
  \begin{proof}
Let $0 < \eps < t < 1$ and we define 
\begin{equation*}
\psi_\eps =
\begin{cases}
0 & \mbox{if}\ \utilp \le t\\[.4cm]
\ds\frac{\utilp}{\eps}\frac{1}{\utilp^{p-1}} & \mbox{if}\ t < \utilp < t+\eps\\
\ds\frac{1}{\utilp^{p-1}} & \mbox{if}\ \utilp \ge t+\eps.
\end{cases}
\end{equation*}
The functions $\psi_\eps$ are in $W^{1,p}(\Omega)$ and increasingly converge to $\utilp^{-(p-1)}\chi_{U_{t}}$ as $\eps\rightarrow 0$.\\
Moreover, we can obtain that
\begin{equation*}
\D\psi_\eps=
\begin{cases}
0 & \mbox{if}\ \utilp < t\\[.4cm]
\ds\frac{1}{\eps}\left((p-1)\ds\frac{t}{\utilp}+2-p\right)\frac{\D \utilp}{\utilp^{p-1}} & \mbox{if}\ t < \utilp < t+\eps\\
-(p-1)\ds\frac{\D \utilp}{\utilp^p} & \mbox{if}\ \utilp > t+\eps.
\end{cases}
\end{equation*}
Then choosing $\psi_\eps$ as test function in \eqref{weak.problem}, we get that the first integral is
\begin{multline*}
-(p-1)\int_{U_{t+\eps}} \frac{[F(\D \utilp)]^{p}}{ \utilp^{p}} dx +
\frac 1 \eps \int_{U_{t}\setminus U_{t+\eps}} \frac{[F(\D \utilp)]^{p}}{\utilp^{p-1}}\left( (p-1)\frac{t}{\utilp}+2-p \right)dx = \\ =
-(p-1)\int_{U_{t+\eps}} \frac{[F(\D \utilp)]^{p}}{ \utilp^{p}} dx + 
\frac 1 \eps \int_{t}^{t+\eps} \left( (p-1)\frac{t}{\tau}+2-p \right) \int_{S_{\tau}} \frac{[F(\D \utilp)]^{p-1}}{\utilp^{p-1}} F(\nu) d\haus,
\end{multline*}
where last equality follows by the coarea formula. Then, reasoning as in \cite{bd10} and \cite{pota} we get that
\[
	\int_{\Omega} [F(\D \utilp)]^{p-1}F_{\xi}(\D \utilp) \cdot \D\psi_{\eps} dx  	\xrightarrow{\eps\rightarrow 0} -(p-1) \int_{U_{t}} \frac{[F(\D \utilp)]^{p}}{\utilp^{p}} dx + \int_{S_{t}} \frac{[F(\D \utilp)]^{p-1}}{\utilp^{p-1}} F(\nu) d\haus.
\]
As regards the other two integrals in \eqref{weak.problem}, we have
\begin{multline*}
\int_{\de\Omega}\beta \utilp^{p-1}\psi_{\eps}F(\nu)\,d\haus =
\int_{\Gamma_{t+\eps}}\beta F(\nu)\,d\haus+ 
\int_{\Gamma_{t}\setminus\Gamma_{t+\eps}}\beta\frac{\utilp-t}{\eps} F(\nu)\,d\haus\xrightarrow{\eps\rightarrow 0}\\
\xrightarrow{\eps\rightarrow 0} \int_{\Gamma_{t}}\beta F(\nu)\,d\haus,
\end{multline*}
by dominated convergence theorem and by monotone convergence theorem and the definition of $\psi_\eps$,
\[
\ell_{1}(\beta, \Omega)\int_{\Omega} \utilp^{p-1} \psi_{\eps}dx \;\;\xrightarrow{\eps\rightarrow 0}\;\; \ell_{1}(\beta , \Omega)|U_{t}|.
\]
Summing the three limits, we get \eqref{rep.ell1}.
	\end{proof}
	When we consider a generic test function we have
	\begin{thm}
	Let $\Omega \subset \R^{n}$ be a bounded open set with $C^{1,\alpha}$ boundary and let $\alpha \in ]0,1[$.
	Let $\varphi$ be a nonnegative function in $\Omega$ such that $\varphi\in L^{p'}(\Omega)$, where $p'=\frac{p}{p-1}$. If $\varphi\not\equiv [F(\D \utilp)]^{p-1}/\utilp^{p-1}$, where $\utilp$ is the eigenfunction given in Theorem \ref{ug_beta_var}, and $\mathcal F_{\Omega}$ is the functional defined in \eqref{F.betavar}, then there exists a set $S\subset ]0,1[$ with positive measure such that for every $t\in S$ it holds that
\begin{equation}
	\label{tesithm}
	\ell_{1}( \beta, \Omega)>\mathcal{ F}_{\Omega}(U_{t},\varphi).
\end{equation}
	\end{thm}
\noindent	The proof is similar to that obtained in \cite{bd10} and \cite{pota}, and we only sketch it here. It can be divided in two main steps. First, we claim that, if
	\[
	w(x):=\varphi-\frac{[F(\D \utilp)]^{p-1}}{\utilp^{p-1}},\qquad I(t):=				\int_{U_{t}} w\frac{F(\D \utilp)}{\utilp}\,dx,
	\] 
	then the function $I: ]0,1[\rightarrow\R$ is locally absolutely continuous and
	\begin{equation}\label{dis_der}
	\mathcal{F}_{\Omega}(U_t , \fhi) \le \ell_1(\beta , \Omega) - \frac{1}{\abs{U_t}t^{p-1}}\left(\frac{d}{dt}\left(t^p I(t)\right)\right)
	\end{equation}
for almost every $t\in ]0,1[$. Second, we show that the derivative $\frac{d}{dt}\left(t^p I(t)\right)$ is positive in a subset of $]0,1[$ with nonzero measure.\\
In order to prove \eqref{dis_der}, using the representation formula \eqref{rep.ell1} we obtain that, for a.e. $t\in ]0,1[$,
 \begin{equation}
\label{secondineq}
\begin{split}
\mathcal F_{\Omega}(U_{t},\varphi)&= \ell_{1}(\beta ,\Omega)+\frac{1}{|U_{t}|}
\left(
	\int_{S_{t}} w F(\nu)\,d\haus-(p-1)\int_{U_{t}}
	\Big(\varphi^{p'}-\frac{[F(\D \utilp)]^{p}}{\utilp^{p}}\Big)dx
\right)\\[.2cm]
&\le 
\ell_{1}(\beta , \Omega)+\frac{1}{|U_{t}|}
\left(
	\int_{S_{t}} w F(\nu)\,d\haus-p\int_{U_{t}}
	 w\frac{F(\D \utilp)}{\utilp}dx
\right)\\[.2cm]
&=
\ell_{1}(\beta , \Omega)+\frac{1}{|U_{t}|}
\left(
	\int_{S_{t}} wF(\nu)\,d\haus-p\, I(t)
\right)
\end{split}
\end{equation}
where the inequality in \eqref{secondineq} follows from the inequality $\fhi^{p^\prime}\ge v^{p^\prime} + p^\prime v^{p^\prime-1}(\fhi-v)$, with $\fhi,\ v \ge 0 $. Proceeding as in \cite{pota} and using the coarea formula we, obtain for a.e. $t\in ]0,1[$
\begin{equation}\label{derivata}
-\frac{d}{dt}(t^p I(t))= t^{p-1}\left(
	\int_{S_{t}} wF(\nu)\,d\haus-p\, I(t)
\right).
\end{equation}
Substituting \eqref{derivata} in \eqref{secondineq} we obtain \eqref{dis_der}. We can conclude the proof, arguing by contradiction exactly as in \cite[Theorem 3.2]{bd10}, indeed is possible to see that the function $t^p I(t)$ has positive derivative in a set of positive measure. This fact with \eqref{dis_der} give us the inequality \eqref{tesithm}. 

%%%%%%%%%%%%%%%%%%%%%%%%%%%%%%%%%%%%%%%%%%%%%%%%%%%%%%%%%%%%%%%%%%%%%
\subsection{Applications}
In this section we use the representation formula given in Theorem \ref{ug_beta_var} in order to get some estimates for $\ell_1(\beta,\Omega).$

\subsubsection{A Faber-Krahn type inequality}
Let $\Omega \subset \R^{n}$ be  a bounded open set with $C^{1,\alpha}$ boundary,  $\alpha \in ]0,1[$ and let $\W_R$ be the  Wulff shape centered at the origin   with radius $R$ such that  $|\Omega|=|\mathcal W_{R}|$.

Let $\bar\beta$ be a positive constant and let us consider the following Robin eigenvalue problem in $\mathcal W_{R}$ for $\Qp$
\begin{equation}
  \label{rad}
  \left\{
    \begin{array}{ll}
      -\Qp v=\ell_1(\bar\beta,\mathcal W_{R}) 
      |v|^{p-2}v 
      &\text{in } 
      \mathcal W_{R},\\[.2cm] 
      % H(Du)^{p-1}H_\xi(D u)
    F^{p-1}(\D v)F_\xi(\D v)\cdot\nu + \bar\beta F(\nu)\vert v\vert^{p-2}v=0
      &\text{on } \de\mathcal W_{R}.
    \end{array}
  \right.
\end{equation}
Let $w(t)$, $t \in [0,+\infty[$, be a non negative continuous function such that 
\begin{equation}
\label{ipbeta1}
w(t) \ge C(R)\, t,
\end{equation}
where $C(R)=\ell_1(\bar\beta,\mathcal W_{R}) +(p-1) \bar \beta^{\frac{p}{p-1}}$ is the constant appearing in \eqref{costc}.

Let us consider the following Robin eigenvalue problem 
 \begin{equation}
  \label{pb}
  \left\{
    \begin{array}{ll}
      -\Qp u=\ell_{1}(\beta,\Omega) |u|^{p-2}u
      &\text{in } 
      \Omega,\\[.3cm] 
      % H(Du)^{p-1}H_\xi(D u)
      F^{p-1}(\D u)F_\xi(\D u)\cdot\nu + \beta(x) F(\nu)\vert u\vert^{p-2}u=0
      &\text{on } \de\Omega, 
    \end{array}
  \right.
\end{equation}
where 
\begin{equation}
\label{ipbeta}
\beta(x)= w(F^o(x)), \quad x \in \de \Omega.
\end{equation}
As a consequence of the representation formula \eqref{ug_beta_var} for $\ell_{1}(\beta, \Omega)$ we get  the following Faber-Krahn inequality.
\begin{thm}
Let $\Omega \subset \R^{n}$ be  a bounded open set with $C^{1,\alpha}$ boundary, let $\alpha \in ]0,1[$ and let $\W_R$ be the  Wulff shape  such that  $|\Omega|=|\mathcal W_{R}|$. Let $w(t)$, $t \in [0,+\infty[$, be a non negative continuous function  which verifies \eqref{ipbeta1} and let $\beta(x)$ be the function defined in \eqref{ipbeta}.
 Then, 
\begin{equation}
\label{fktesi}
 \ell_{1}(\bar\beta,\mathcal W_{R})\le \ell_{1}(\beta,\Omega).
\end{equation}  
\end{thm}
\begin{proof}
We construct a suitable test function in $\Omega$ for \eqref{tesithm}.
Let $v_{p}$ be a positive eigenfunction of the  radial problem \eqref{rad} in $B_{R}$. By Theorem \ref{teorad}, $v_{p}$ is a function depending only by $F^{o}(x)$, $v_{p}=\rho_p(F^o(x))$, and then we can argue as in Section 4  defining  the function 
\[
	f(r_{x})=\varphi_{\star}(x) = \frac{\left[F\big(\D v_p(x)\big)\right]^{p-1}F_\xi\big(Dv_p(x)\big)\cdot\nu}{{v_{p}(x)}^{p-1}F(\nu)} = \frac{(-\rho_p^\prime(r_{x}))^{p-1}}{(\rho_p(r_{x}))^{p-1}},
	\]
where $r_{x}=F^{o}(x) \in [0,R]$
%We recall that the function $f(r)$ is increasing respect to $r$.

Denoted by $\mathcal W_{s}=\{x \in \mathcal W_{R} \colon v_{p}(x)>s\}, \, 0<s<R, $  clearly $\mathcal W_{s}$ is a Wulff shape centered at the origin and by Theorem \eqref{ug_beta_var} we get
\begin{equation}
\label{firad}
\ell_{1}(\beta(R), \mathcal W_{R})=F_{\mathcal W_{R}}(\mathcal W_{s},\varphi_{\star})=\frac{1}{| \mathcal W_{s}|}\left( -(p-1)\int_{\mathcal W_{s}} \varphi_{\star}^{p'}dx + 
	\int_{\de \mathcal W_{s}} \varphi_{\star} F(\nu) d \mathcal H^{n-1}
	\right)
\end{equation}

 Let  $\tilde u_{p}$ be the first eigenfunction of \eqref{pb} in 
$\Omega$ such that $\|\tilde u_{p}\|_{\infty}=1$. 
For $x\in\Omega$ we set $\tilde u_{p}(x)=t$, $0<t<1$.
Then we consider the Wulff shape $\mathcal W_{r(t)}$, centered at the origin, where $r(t)$ is the positive number such that $|U_{t}|=|\mathcal W_{r(t)}|$. Then,  we define the following test function
\[
	\varphi(x):=f(r(t))=f(F^o(x)).
\]
We stress that clearly $r(t)<R.$
 Our aim is to compare $\mathcal F_{\Omega}(U_{t},\varphi)$ with $\mathcal F_{\mathcal W_{R}}(\mathcal W_{r(t)},\varphi_{\star})$. Then by \eqref{firad} with $s=r(t)$ we have to show that 
\begin{equation*}
\begin{split}
	\label{compar}
	\mathcal F_{\Omega}(U_{t},\varphi) & \ge
 \frac{1}{|\mathcal W_{r(t)}|}
	\left( -(p-1)\int_{\mathcal W_{r(t)}} \varphi_{\star}^{p'}dx + 
	\int_{\de  \mathcal W_{r(t)}} \varphi_{\star}F(\nu)d \mathcal H^{n-1}
	\right)
	 \\ &=
	\mathcal F_{\mathcal W_{R}}(\mathcal W_{r(t)},\varphi_{\star}). 
\end{split}
\end{equation*}

We first observe that by \cite[Section 1.2.3]{maz}, being $|U_{t}|=|\mathcal W_{r(t)}|$ for all $t\in ]0,1[$
\begin{equation*}
\label{eqpr}
	\int_{U_{t}} \varphi^{p'}dx =\int_{\mathcal W_{r(t)}} \varphi_{\star}^{p'}dx.
\end{equation*}
Moreover,  from the  weighted isoperimetric inequality quoted in Remark \ref{rwii},  Theorem \ref{confr}  and the assumption \eqref{ipbeta} on $\beta$ we get

\begin{multline*}
	\int_{\de \mathcal W_{r(t)}} \varphi_{\star} F(\nu)d\sigma= \int_{\de \mathcal W_{r(t)}} f(r(t)) F(\nu)d \mathcal H^{n-1} \le \int_{\de U_{t}} f(F^o(x)) F(\nu)d \mathcal H^{n-1} \le \\ \le \int_{S_{t}} f(F^o(x)) F(\nu)d \mathcal H^{n-1}+\int_{\Gamma_{t}} f(F^o(x)) F(\nu)d \mathcal H^{n-1} \\= \int_{S_{t}} \varphi F(\nu)d \mathcal H^{n-1} + \int_{\Gamma_{t}} f(F^o(x)) F(\nu)d \mathcal H^{n-1}	\\ \le	\int_{S_{t}} \varphi F(\nu)d \mathcal H^{n-1} + C(R)\int_{\Gamma_{t}} F^o(x) F(\nu)d \mathcal H^{n-1}\\ \le \int_{S_{t}} \varphi F(\nu)d \mathcal H^{n-1} + \int_{\Gamma_{t}}w( F^o(x)) F(\nu)d \mathcal H^{n-1}
\end{multline*}
and this concludes the proof.
\end{proof}
\begin{rem}
When $\beta =\bar\beta$ is a nonnegative constant \eqref{fktesi} is proved in \cite{pota} in the anisotropic case and in \cite{bd10,df} in the Euclidean case. 
\end{rem}
%%%%%%%%%%%%%%%%%%%%%%%%%%%%%%%%%%%%%%%%%%%%%%%%%%%%%%%%%%%%%%%%%
\subsubsection{A Cheeger type inequality  for $\ell_1(\beta,\Omega)$}
In this part we introduce the anisotropic weighted Cheeger constant and, using the representation formula we prove an anisotropic weighted Cheeger inequality for $\ell_{1}(\beta, \Omega)$. Following \cite{cas} we give
%In order to define the weighted anisotropic perimeter for subsets $ E \subset \Omega$, we need to require that $\beta$ is a continuous function defined in the whole $\overline \Omega.$
\begin{defn}\label{awCHEE}
Let $g \colon \overline \Omega \to ]0,\infty[$ be a continuous function  the weighted anisotropic Cheeger constant is defined as follows
$$
h_{g}(\Omega)=\inf_{U\subset\Omega} \frac{\ds\int_{\de U} g F(\nu) d\haus}{\abs{U}} =\inf_{U\subset\Omega} \displaystyle \frac{P_{F}(g,U)}{\abs{U}}.
$$ 
\end{defn}
We observe that when $g(x)=c$ is a constant then 
\begin{equation*}
\label{ccc}
h_{g}(\Omega)= c \inf_{U\subset\Omega}\frac{P_{F}(U)}{\abs{U}}=c \, h(\Omega) ,
\end{equation*}
where $h(\Omega) $ is the anisotropic Cheeger constant defined in \eqref{cc} .
In \cite{cas} it is proved that actually $h_{g}(\Omega)$ is a minimum that is there exists a set $C\subset \Omega$ such that
\[
h_{g}(\Omega)=\displaystyle \frac{P_{F}(g,C)}{\abs{C}},
\]
and we refer to $C$ as a weighted Cheeger set.

We observe that for suitable weight $g$ the constant $h_g(\Omega)$ verifies an anisotropic isoperimetric inequality
\begin{thm}
Let $g(x)=w(F^o(x))=w(r)$, $r\ge 0$ with $w$ a non negative and nondecreasing function such that 
\begin{equation*}
w(r^{\frac 1 n}) r^{1-\frac{1}{n}} , \qquad  0 \le r \le R^n, 
\end{equation*}
is convex with respect to $r$. Then 
\[
h_g(\Omega) \ge h_g(\mathcal W_R)=\frac{nw(R)}{R},
\]
where $\mathcal W$ is a Wulff shape with the same measure as $\Omega$.
\end{thm}
\begin{proof}
The proof follows immediately from Remark \ref{rwii}.
\end{proof}
 When $\beta =\bar\beta$ is a nonnegative constant and $p=2$ in \cite{jk} the following Cheeger inequality is proved in the Robin eigenvalue case
\begin{equation}
\label{ck}
\ell_1( \bar\beta , \Omega )\ge
\begin{cases}
h(\Omega)\bar\beta-\bar\beta^{2} & \mbox{always}\\[.3cm]
\ds\frac{1}{4}\left[ h(\Omega) \right]^{2} & \mbox{if}\ \bar\beta \ge \ds\frac{1}{2} h(\Omega)
\end{cases}
\end{equation}
In the next result we extend \eqref{ck} to the anisotropic case for any $1<p<\infty$ considering  $\beta$ not in general constant. 

\begin{thm}\label{in_aw_CHEE}
Let us consider problem \eqref{bvp.var.beta} with $\beta\in C(\overline\Omega)$ such that $\beta\ge0$. Then the following weighted anisotropic Cheeger inequality holds 
\begin{equation}
\label{caw}
\ell_1(\beta, \Omega)\geq h_{\beta}(\Omega) - (p-1)\|\beta^{p^\prime}\|_{L^{\infty}(\overline\Omega)}, 
\end{equation}
where $p'=\frac{p}{p-1}$.
%where $h_{\beta}(\Omega)$ is the weighted anisotropic Cheeger  constant defined above. 
\end{thm}
\begin{proof}
Using $\beta$ as test function in \eqref{tesithm} we obtain
\begin{multline*}
\ell_1(\beta , \Omega) \ge \mathcal{F}(U_t , \beta) = \ds\frac{1}{\abs{U_t}}\left(-(p-1)\int_{U_t}\beta^{p^\prime} dx + \int_{S_t} \beta F(\nu) d\haus + \int_{\Gamma_t} \beta F(\nu) d\haus\right)\\
= \ds\frac{1}{\abs{U_t}}\left(-(p-1)\int_{U_t}\beta^{p^\prime} dx + \int_{\de U_t} \beta F(\nu) d\haus\right) \ge -(p-1)\|\beta^{p^\prime}\|_\infty + h_{\beta}(\Omega).
\end{multline*}
\end{proof}
\begin{rem}
We observe that the previous result continues to hold if we take $\beta \in C(\de \Omega)$. Indeed in this case from a classical result, see for instance \cite[Theorem 4.I]{g57}, we know that the function $\beta $ is the trace of a nonnegative function $\beta_\Omega \in C\left(\overline{\Omega}\right)$. Then inequality  \eqref{caw} holds with $\beta = \beta_\Omega.$
\end{rem}
We emphasize the inequality \eqref{caw} in the particular case of $\beta =\bar\beta$ is a nonnegative constant. 
\begin{cor}\label{corche}
The first eigenvalue $ \ell_1( \bar\beta , \Omega ) $ of \eqref{bvp.var.beta} on a fixed bounded open set $\Omega \subset \R^n$ with Lipschitz boundary satisfies
\begin{equation}\label{cheeger}
\ell_1( \bar\beta , \Omega )\ge
\begin{cases}
h(\Omega)\bar\beta-(p-1)\bar\beta^{\frac{p}{p-1}} & \mbox{always}\\[.3cm]
\ds\frac{1}{p^{p}}\left[ h(\Omega) \right]^{p} & \mbox{if}\ \bar\beta \ge \ds\frac{1}{p^{p-1}} \left[h(\Omega)\right]^{p-1}
\end{cases}
\end{equation}
\end{cor}
\begin{proof}
From the Theorem \ref{in_aw_CHEE} we obtain, using the constant function $\bar\beta$ as test, we obtain the first part of the inequality.
For the second part, is suitable using as test function in the functional $\mathcal{F}_\Omega(U_t, \cdot)$, the constant $\frac{1}{p^{p-1}} \left[h_F(\Omega)\right]^{p-1}$ under the assumption that the constant $\bar\beta \ge \frac{1}{p^{p-1}} \left[h_F(\Omega)\right]^{p-1}$.
\end{proof}
%\textcolor{red}{ Aggiungere stima dall'alto}

%\textcolor{red}{ Stime di h in termini di beta}
\begin{rem}
From the anisotropic Cheeger inequality for constant $\bar\beta$ we obtain immediately a lower bound for $\ell_1(\bar\beta,\Omega)$ in terms of the anisotropic inradius of $\Omega$  different from \eqref{st_pota} by using \eqref{hr}.
\end{rem}

\begin{rem}
By $(ii)$ Theorem \ref{th:lbdpr2} and Corollary \ref{corche} we obtain for $\bar\beta \ge \frac{1}{p^{p-1}}\left[h(\Omega)\right]^{p-1}$ the anisotropic Cheeger inequality \eqref{cheeger} for the first Dirichlet eigenvalue of $\Qp$
\begin{equation*}
\lambda_D(\Omega) \ge \ell_1(\bar\beta , \Omega) \ge \frac{1}{p^p}\left[h(\Omega)\right]^p.
\end{equation*}
\end{rem}

\end{document}